\documentclass[reqno,psamsfonts,11pt]{gen-j-l}
\usepackage{amssymb,amsmath,amsbsy,amsthm,esint,setspace}
\usepackage{hyperref}
\usepackage{amsrefs}
\usepackage{enumitem}
\usepackage{color}

\newtheorem{theorem}{Theorem}[section]
\newtheorem{lemma}[theorem]{Lemma}
\newtheorem{proposition}[theorem]{Proposition}
\newtheorem{corollary}[theorem]{Corollary}

\newtheorem*{remark}{Remark}

\let\temp\phi
\let\phi\varphi
\let\varphi\temp

\begin{document}

\title[Observables and Strong One-Sided Chaos]
{Observables and Strong One-Sided Chaos in the Boltzmann-Grad Limit}

\author{Ryan Denlinger}

\begin{abstract}
Boltzmann's equation provides a microscopic model for the evolution of
dilute classical gases.
 A fundamental problem in
mathematical physics is to rigorously derive Boltzmann's equation 
starting from Newton's laws. In the 1970s,
Oscar Lanford provided such a derivation, for the hard sphere interaction,
on a small time interval. One of the subtleties of Lanford's original proof
was that the strength of convergence proven
 at positive times was much weaker than
that which had to be assumed at the initial time, which is at odds with
the idea of propagation of chaos. Several authors have addressed this
situation with various notions of strong one-sided chaos, which is the true
property which is propagated by the dynamics. We provide a new approach
to the problem based on duality and the evolution of observables; the
observables encode the detailed interaction and allow us to define a
new notion of strong one-sided chaos.  
\end{abstract}

\maketitle

\section{Introduction}
\label{intro}

Kinetic theory is concerned with the description of dilute gases at the microscopic level.
The fundamental equation of collisional kinetic theory is Boltzmann's equation, which is an
evolutionary partial differential equation (PDE) stated in terms of the phase-space
density $f(t,x,v)\geq 0$. Boltzmann's equation
describes the effects of free transport and binary collisions between particles.
In the case of hard spheres, Boltzmann's equation is written 
\begin{equation}
\label{eq:IN-boltz}
\left( \partial_t + v \cdot \nabla_x \right) f(t,x,v) = Q(f,f) (t,x,v)
\end{equation}
where $Q$ is the \emph{collision operator} 
\begin{equation}
Q(f,f) = Q^+ (f,f) - Q^- (f,f)
\end{equation}
\begin{equation}
Q^+ (f,f) = \int_{\mathbb{R}^d \times \mathbb{S}^{d-1}}
\left[ \omega \cdot (v_2-v)\right]_+ 
f(t,x,v^*) f(t,x,v_2^*) dv_2 d\omega
\end{equation}
\begin{equation}
Q^- (f,f) = \int_{\mathbb{R}^d \times \mathbb{S}^{d-1}}
\left[ \omega \cdot (v_2-v)\right]_+ 
f(t,x,v) f(t,x,v_2) dv_2 d\omega
\end{equation}
and
\begin{equation}
\begin{aligned}
v^* & = v + \omega \omega \cdot (v_2 - v) \\
v_2^* & = v_2 - \omega \omega \cdot (v_2 - v) 
\end{aligned}
\end{equation}
We refer to \cite{CIP1994} for a mathematical introduction to Boltzmann's equation.

One of the central problems in kinetic theory is to derive Boltzmann's equation starting
from
a system of $N$ particles interacting via a classical Hamiltonian. For
the hard sphere Boltzmann equation, the interaction is given by the hard sphere
potential (billiard balls). The scaling limit in which one derives Boltzmann's
equation is known as the \emph{Boltzmann-Grad limit}; in this limit,
we have $N$ identical hard spheres of diameter $\varepsilon$, with
$N\varepsilon^{d-1} = \ell^{-1}$ for some fixed $\ell > 0$. 
However, the Boltzmann-Grad scaling by itself \emph{does not} force a Boltzmann type
dynamics; it is also necessary to assume that the configurations of distinct particles
are independent from one another at some initial time.
Unfortunately, the independence between particles \emph{is not} propagated by the
$N$ particle dynamics. If independence breaks down at some positive time then the
validation of Boltzmann's equation is expected to fail.
 For this reason, one of the
key problems in validating Boltzmann's equation is to specify an appropriate notion of
independence, or factorization, among particles.\footnote{More generally, it is possible
to work with a notion of \emph{exchangeability} in place of \emph{independence}.
In this case one derives \emph{statistical solutions} of Boltzmann's equation.}

Lanford has shown, in the 1970s, that if the particles are asymptotically decorrelated in
a \emph{very strong sense} at the initial time, then a \emph{weaker} notion of
factorization is propagated to positive times. \cite{L1975,GSRT2014} Indeed, Lanford
assumes that the marginal associated to $s$ particles (for $s$ fixed) converges in
$L^\infty$ at the initial time. This type of convergence \emph{cannot} be propagated
to positive times, for the following reason. If the marginals converge in $L^\infty$
at some positive time, then all we need to do is reverse the particle velocities and
evolve forwards to again deduce Boltzmann's equation on the time interval $[t,t+\tau)$.
This is because $L^\infty$ is invariant under reversal of the velocities of all the particles.
 But if
we reverse the particle velocities and evolve forwards then we should actually obtain
the \emph{backwards Boltzmann equation} (this is Boltzmann's equation with 
a minus sign in front of $Q(f,f)$),
because this is the only possibility which is consistent with the evolution we have
already deduced on $[0,t]$. (This is due to the reversibility of 
Newton's laws.) We refer to \cite{CIP1994} for a detailed discussion of the issue of
irreversibility.

More generally, \emph{any} norm which is invariant under reversal
of particle velocities cannot be the correct norm for proving propagation of chaos,
by the same argument. The essential conflict is that Newton's laws are time-reversible
whereas Boltzmann's equation is irreversible (as evidenced by the \emph{H-theorem}).
This problem leads to the notion of \emph{strong one-sided chaos} (strong chaos),
which means that
 the topology of convergence in Lanford's theorem should be sensitive only to
particle configurations \emph{coming into} a collision. (Indeed, after a collision,
it is impossible for the particles to be completely decorrelated.) This already indicates that
strong chaos must be a very subtle notion, because it implies that the dynamics
 is determined by the values of functions evaluated along
\emph{very small subsets} of their domain of definition.\footnote{This is because the
typical length scale for a collision is of order $\varepsilon$.}  To make matters worse,
if we condition on the event that a collision \emph{does} happen, then the set of
points coming \emph{into} the collision is locally of the same measure as the set of points
going \emph{out} of the collision. This means that we cannot rely on abstract measure-theoretic
notions of regularity; we \emph{must} account for the details of the collision process.

The delicacy of the Boltzmann-Grad limit is the very reason that Lanford required
the $L^\infty$ convergence of the data in the first place. It seems difficult to
provide convergence \emph{at all scales}, as Lanford's theorem requires, without some
type of $L^\infty$ norm. Multiple authors have refined Lanford's theorem
into a strong chaos result by proving uniform convergence along suitiable subsets of the
reduced phase space. We refer to \cite{BGSRS2016,Ki1975,De2017,PSS2014} and Appendix A
of \cite{vBLLS1980} for a variety of results concerning uniform convergence and
strong chaos. The goal of the present work is to provide a new notion of strong chaos which
is much weaker than those which have previously appeared in the literature.
Instead of defining chaos via convergence relative to a single norm, we introduce
a sequence of seminorms which capture information at different scales.
Perhaps the most natural way to understand this new notion
of chaos is to view it as the $L^1$ norm restricted to certain singular sets which
arise from the dynamics.\footnote{We thank Mario Pulvirenti
for pointing this out to us.}

\textbf{Organization of the paper.} In Section \ref{sec:Ntn}, we define basic notation
following \cite{GSRT2014}. 
In Section \ref{sec:BBGKY}, we define the BBGKY and dual BBGKY hierarchies, and
quote several well-posedness results.
We prove a comparison principle for solutions of the dual BBGKY hierarchy in
Section \ref{sec:Comp}. In Section \ref{sec:Sing}, we use the dual BBGKY hierarchy
to define singular sets which are associated with the hard sphere dynamics.
In Section \ref{sec:TR} we develop a connection between solutions of the dual
BBGKY hierarchy and the traditional notion of a pseudo-trajectory.
Finally, in Section \ref{sec:PR}, we state and prove our main result on the
propagation of chaos.

\section*{Acknowledgements}
\label{sec:acknowledgements}

R.D. gratefully acknowledges support from a postdoctoral fellowship
at the University of Texas at Austin. The initial inspiration for
this paper came from a discussion with my PhD advisor, Nader Masmoudi,
at the 2014 Fields Medal Symposium in Toronto.
 Additional thanks go to
Nata{\v s}a Pavlovi{\' c}, Thomas Chen, Mario Pulvirenti,
Laure Saint-Raymond, and Sergio Simonella, for helpful
discussions and comments on the subject of strong chaos; and,
to Herbert Spohn for pointing out a result from Appendix A
of \cite{vBLLS1980}. Any errors in this work are mine alone.

\section{Notation}
\label{sec:Ntn}

We will consider $N$ identical
 hard spheres in the spatial domain $\mathbb{R}^d$, $d\geq 2$.
The spheres all have diameter $\varepsilon > 0$, their centers are located at
positions $x_1,\dots,x_N \in \mathbb{R}^d$, and their velocities are given by
$v_1,\dots,v_N \in \mathbb{R}^d$. The full configuration is written
$Z_N = (X_N,V_N) = (x_1,\dots,x_N,v_1,\dots,v_N)\in \mathbb{R}^{dN}\times
\mathbb{R}^{dN}$. The $N$ particle phase space is specified by the condition of
exclusion of hard spheres:
\begin{equation}
\mathcal{D}_N = \left\{\left. Z_N = (X_N,V_N) \in \mathbb{R}^{dN} \times 
\mathbb{R}^{dN} \right| \forall 1\leq i < j \leq N, \;
|x_i - x_j | > \varepsilon
\right\}
\end{equation}
As long as particles remain within $\mathcal{D}_N$, they continue along free trajectories:
$\dot{X}_N = V_N$, $\dot{V}_N = 0$. If particles collide (i.e. $Z_N \in \partial
\mathcal{D}_N$) then the velocities of the particles are transformed by the law of
specular reflection, then the free evolution continues.
The collective flow of $N$ hard spheres is written
$\psi_N^t : \mathcal{D}_N \rightarrow \mathcal{D}_N$; this is a well-defined measurable
measure-preserving map. \cite{Al1975,GSRT2014} We always enforce the Boltzmann-Grad
scaling $N\varepsilon^{d-1} = \ell^{-1}$ for a fixed parameter $\ell > 0$. We will also
require the reduced phase space $\mathcal{D}_s$ for $1\leq s \leq N$, which is the phase
space of $s$ identical hard spheres of diameter $\varepsilon$; the map 
$\psi_s^t : \mathcal{D}_s \rightarrow \mathcal{D}_s$ defines the flow of $s$ hard spheres.
If $Z_s \in \partial \mathcal{D}_s$ then we denote by $Z_s^*$ the image of $Z_s$ under
the collisional change of variables. In particular, if $x_j = x_i+\varepsilon \omega$
then $Z_s^* = (x_1,v_1,\dots,x_i,v_i^*,\dots,x_j,v_j^*,\dots,x_s,v_s)$ where
\begin{equation*}
v_i^* = v_i + \omega \omega \cdot (v_j-v_i)
\end{equation*}
\begin{equation*}
v_j^*  = v_j - \omega \omega \cdot (v_j-v_i)
\end{equation*}

We shall consider
an initial density $f_N (0) \in L^1 (\mathcal{D}_N)$, which is non-negative,
symmetric under particle interchange, and
normalized so that $\int_{\mathcal{D}_N} f_N (0,Z_N) dZ_N = 1$. The evolved state
$f_N (t)$, $t\geq 0$, is defined as the pushforward of $f_N (0)$ through $\psi_N^t$.
Since the hard sphere flow is measure-preserving, this means that
\begin{equation}
f_N (t,Z_N) = f_N (0,\psi_N^{-t} Z_N)
\end{equation}
We extend $f_N (t)$ by zero to be defined on all of $\mathbb{R}^{dN} \times\mathbb{R}^{dN}$.
Since $f_N (0)$ is symmetric under particle interchange, $f_N (t)$
must be symmetric as well. We define the marignals $f_N^{(s)} (t)$ for $1\leq s \leq N$
by the formula
\begin{equation}
f_N^{(s)} (t,Z_s) = \int_{\mathbb{R}^{d(N-s)} \times \mathbb{R}^{d(N-s)}}
f_N (t,Z_N) dz_{s+1} \dots dz_N
\end{equation}
The symmetry of $f_N$ implies that it does not matter which particles we integrate out,
and the marginals are likewise symmetric under particle interchange.
We also define 
\begin{equation}
E_s (Z_s) = \frac{1}{2} \sum_{i=1}^s |v_i|^2
\end{equation}
\begin{equation}
I_s (Z_s) = \frac{1}{2} \sum_{i=1}^s |x_i|^2
\end{equation}

\section{The BBGKY and Dual BBGKY Hierarchies}
\label{sec:BBGKY}

Let $f_N (t)$ be a solution of Liouville's equation (for the hard sphere
interaction), with marginals
$f_N^{(s)} (t)$ for $1 \leq s \leq N$. Then it is possible to show that the
marginals obey the following hierarchy of equations, known as the BBGKY
hierarchy (Bogoliubov-Born-Green-Kirkwood-Yvon):
\begin{equation}
\left( \partial_t + V_s \cdot \nabla_{X_s}\right) f_N^{(s)} (t,Z_s)
= (N-s)\varepsilon^{d-1} C_{s+1} f_N^{(s+1)} (t,Z_s)
\end{equation}
Here $f_N^{(s)} (t,Z_s^*) = f_N^{(s)} (t,Z_s)$ and the operator
$C_{s+1}$ is defined as follows:
\begin{equation}
C_{s+1} = \sum_{i=1}^s \left(
C_{i,s+1}^+ - C_{i,s+1}^- \right)
\end{equation}
\begin{equation}
\begin{aligned}
& C_{i,s+1}^- f_N^{(s+1)} =
 \int_{\mathbb{R}^d \times \mathbb{S}^{d-1}}
dv_{s+1} d\omega
\left[ \omega \cdot (v_{s+1}-v_i)\right]_-  \times \\
&\qquad \qquad \qquad \qquad 
 \times f_N^{(s+1)} ( t,x_1,v_1,\dots,x_i,v_i,\dots,x_s,v_s,
x_i+\varepsilon \omega,v_{s+1})
\end{aligned}
\end{equation}
\begin{equation}
\begin{aligned}
& C_{i,s+1}^+ f_N^{(s+1)} =
 \int_{\mathbb{R}^d \times \mathbb{S}^{d-1}}
dv_{s+1} d\omega
\left[ \omega \cdot (v_{s+1}-v_i)\right]_+  \times \\
&\qquad \qquad \qquad \qquad 
 \times f_N^{(s+1)} ( t,x_1,v_1,\dots,x_i,v_i^*,\dots,x_s,v_s,
x_i+\varepsilon \omega,v_{s+1}^*)
\end{aligned}
\end{equation}
We remark that it is possible to consider solutions of the BBGKY hierarchy which
do not arise as a consistent sequence of marginals. Such solutions are not
necessarily physically meaningful (e.g. in general the BBGKY hierarchy does
not preserve non-negativity of solutions). However, we will need to consider
general solutions of the BBGKY hierarchy in order to define the so-called
\emph{dual BBGKY hierarchy}. We always assume that the functions 
$f_N^{(s)}$ are symmetric under particle interchange.

Let $\Phi_N = \left\{ \phi_N^{(s)} \right\}_{1\leq s \leq N}$ be a sequence of
real-valued 
functions such that $\phi_N^{(s)}$ is defined on $\mathcal{D}_s$ and symmetric
under particle interchange. Furthermore,
let $F_N = \left\{ f_N^{(s)} \right\}_{1\leq s \leq N}$ be a sequence of
densities, again symmetric under particle interchange. Following
\cite{Ge2013}, we introduce the following duality bracket:
\begin{equation}
\label{eq:duality-1}
\left< \Phi_N,F_N\right> =
\sum_{s=1}^N \frac{1}{s!} \int_{\mathcal{D}_s}
 \phi_N^{(s)} (Z_s) f_N^{(s)} (Z_s) dZ_s 
\end{equation}
The dual BBGKY hierarchy is the evolution equation satisfied by the sequence
of functions $\Phi_N (t)$ under the condition that, for any solution 
$F_N (t)$ of the BBGKY hierarchy, the following holds
\begin{equation}
\label{eq:duality-2}
\left< \Phi_N (t),F_N (0) \right> =
\left< \Phi_N (0),F_N (t) \right>
\end{equation}
It is possible to show, by elementary computation, that the dual BBGKY hierarchy
for hard spheres is given by the following sequence of equations:
\begin{equation}
\label{eq:BBGKY-dual-1}
\left( \partial_t - V_s \cdot \nabla_{X_s} \right) \phi_N^{(s)} (t,Z_s) = 0
\qquad  \left( Z_s \in \mathcal{D}_s, \; s = 1,\dots,N \right)
\end{equation}
\begin{equation}
\label{eq:BBGKY-dual-2}
\begin{aligned}
\frac{\phi_N^{(s)}(t,Z_s^*)}{N-s+1} & + \phi_N^{(s-1)}(t,(Z_s^*)^{(i)})
+ \phi_N^{(s-1)}(t,(Z_s^*)^{(j)}) = \\
& =  \frac{\phi_N^{(s)}(t,Z_s)}{N-s+1}
 + \phi_N^{(s-1)}(t,Z_s^{(i)}) + \phi_N^{(s-1)}(t,Z_s^{(j)})
\end{aligned}
\end{equation}
\begin{equation*}
\quad \quad \quad \quad \quad \quad \quad \quad \quad 
 \left(Z_s\in\left(\Sigma_s (i,j)\times \mathbb{R}^{d s}\right)
\cap \partial \mathcal{D}_s,
\; s = 2,\dots,N\right)
\end{equation*}
Here $Z_s^{(i)} = (z_1,\dots,z_{i-1},z_{i+1},\dots,z_s)$. The dual
BBGKY hierarchy may be solved by applying induction on $s$ and the
method of characteristics. The value of $\phi_N^{(s)}$ is transported
freely along characteristics between collisions. Across any collision,
the value of $\phi_N^{(s)}$ experiences a ``jump'' which is determined
by $\phi_N^{(s-1)}$. Since the evolution of $\phi_N^{(1)}$ is trivially
determined, it is likewise possible to determine every
$\phi_N^{(s)}$ inductively using Duhamel's formula. We will sometimes
refer to solutions of the dual BBGKY hierarchy as
\emph{observables}.

We will quote two local well posedness theorems for the dual BBGKY
hierarchy. Both theorems are proven in \cite{De2017}, though the
proofs are based heavily on classical proofs of well-posedness
for the BBGKY hierarchy. \cite{L1975,IP1986,IP1989} 
We refer the reader to \cite{L1975,Ki1975,PSS2014,GSRT2014,De2017}
for standard well-posedness results concerning the BBGKY hierarchy.
Given parameters $\beta > 0$, $\mu \in \mathbb{R}$, let us define
the norm
\begin{equation}
\left\Vert \Phi_N \right\Vert_{\mathcal{L}^1_{\beta,\mu}}
= \sum_{s=1}^N \frac{1}{s!}
\int_{\mathcal{D}_s}
\left| \phi_N^{(s)} (Z_s)\right| e^{-\beta E_s (Z_s)}
e^{-\mu s} dZ_s
\end{equation}

\begin{theorem}
\label{thm:dual-LWP}
Suppose $\Phi_N (0) = \left\{ \phi_N^{(s)} (0,Z_s) \right\}_{1\leq s \leq N}$
is a sequence of functions, with each $\phi_N^{(s)} (0)$ defined on
$\mathcal{D}_s$ and symmetric under particle interchange. Furthermore,
suppose that $\left\Vert \Phi_N (0) \right\Vert_{\mathcal{L}^1_{\frac{1}{2}\beta,(\mu-1)}}
< \infty$ for some $\beta > 0$, $\mu \in \mathbb{R}$. There exists a constant
$C_d > 0$, depending only on $d$, such that if
$T_L < C_d \ell e^{\mu} \beta^{\frac{d+1}{2}}$, then there exists a unique
solution $\Phi_N (t)$ of the dual BBGKY hierarchy for $t \in [0,T_L]$ 
satisfying the bound
\begin{equation}
\sup_{0 \leq t \leq T_L}
\left\Vert \Phi_N (t) \right\Vert_{\mathcal{L}^1_{\beta,\mu}} \leq
\left\Vert \Phi_N (0) \right\Vert_{\mathcal{L}^1_{\frac{1}{2}\beta,(\mu-1)}}
\end{equation}
\end{theorem}

\begin{theorem}
\label{thm:dual-LWP-2}
If $\ell^{-1} e^{-\mu} \beta^{-\frac{d+1}{2}}$ is sufficiently small
(depending only on $d$) then for any $T>0$ we have the following:
Suppose $\Phi_N (0)= \left\{ \phi_N^{(s)} (0,Z_s) \right\}_{1\leq s \leq N}$ is a
sequence of functions, with each $\phi_N^{(s)} (0)$ defined on $\mathcal{D}_s$
and symmetric under particle interchange. Further suppose that, for some
$\beta > 0$ and $\mu \in \mathbb{R}$, there holds
\begin{equation}
\sum_{s=1}^N \frac{1}{s!} \int_{\mathcal{D}_s}
\left| \phi_N^{(s)} (0,Z_s) \right|
e^{-\frac{1}{2} \beta \left( E_s (Z_s) + I_s ((X_s - T V_s,V_s))\right)}
e^{-(\mu-1) s} dZ_s < \infty
\end{equation}
Then the dual BBGKY hierarchy has a unique solution for 
$t \in [0,T]$ and it satisfies
\begin{equation}
\begin{aligned}
& \sum_{s=1}^N \frac{1}{s!}
\int_{\mathcal{D}_s} \left| \phi_N^{(s)} (T,Z_s) \right|
e^{-\beta (E_s (Z_s) + I_s (Z_s))} e^{-\mu s} dZ_s \\
& \qquad \leq \sum_{s=1}^N \frac{1}{s!}
\int_{\mathcal{D}_s} \left| \phi_N^{(s)} (0,Z_s) \right|
e^{-\frac{1}{2} \beta (E_s (Z_s) + I_s ((X_s - T V_s,V_s)))}
e^{-(\mu-1)s} dZ_s
\end{aligned}
\end{equation}
\end{theorem}

\begin{remark}
In the context of the original paper by Illner \& Pulvirenti 
\cite{IP1986,IP1989}, the smallness condition in Theorem
\ref{thm:dual-LWP-2} is viewed as a largeness condition on the
mean free path $\ell$ (relative to $f_0$), which in turn regulates
the magnitude of
the nonlinearity. In this sense,
Theorem \ref{thm:dual-LWP-2} is a ``global-in-time'' result which
is valid for ``small'' initial data (i.e. small $f_0$).
However, it is also possible
to view $\ell$ as fixed, choose an \emph{arbitrary} (compactly
supported) initial observable $\Phi_N (0)$,
and then choose values of $\beta,\mu$ which meet the smallness
condition; this is always possible because $\mu$ ranges over
all of $\mathbb{R}$. This gives us a way to make global sense
of the dual BBGKY hierarchy for arbitrary observables. It is
important to realize that this \emph{does not} allow us to
relax the small time condition in Lanford's theorem (for large
$f_0$) because not
all observables can be paired against a given $f_0$ to yield
a finite duality pairing. In order to make effective use of the
global boundedness of observables, it is absolutely necessary to
understand cancellation effects; such an understanding is well
out of reach at the present time.
\end{remark}

\section{A Comparison Principle}
\label{sec:Comp}

The dual BBGKY hierarchy does not preserve positivity of solutions.
It is easy to see that typical non-negative data
$\left\{ \phi_N^{(s)} (0) \right\}_{1\leq s \leq N}$ will lead to
solutions $\left\{ \phi_N^{(s)} (t) \right\}_{1\leq s \leq N}$ that
cease to be non-negative for $t>0$.
Moreover, typical initial conditions lead to solutions that
strongly concentrate on very small subsets of the phase space.
 On the other hand, very
special initial conditions lead to trivial evolutions; for instance,
we can let $\phi_N^{(s)} (0) = \mathbf{1}_{\mathcal{D}_s}$ 
identically for all $s$, then
$\phi_N^{(s)} (t) = \mathbf{1}_{\mathcal{D}_s}$ for all $t>0$ and
all $s$.
 The goal of the present section is to construct an
alternative hierarchy which controls the dual BBGKY hierarchy
pointwise but which has better monotonicity properties.

The first step is to construct a \emph{lower envelope}
$\underline{\phi}_N^{(s)}$ and an \emph{upper envelope}
$\overline{\phi}_N^{(s)}$, which will control the dual BBGKY
hierarchy from below and above. The lower and upper envelopes are,
in turn, coupled to each other. The upper envelope solves the
following evolution equation:
\begin{equation}
\left( \partial_t - V_s \cdot \nabla_{X_s} \right) 
\overline{\phi}_N^{(s)} (t,Z_s) = 0
\qquad  \left( Z_s \in \mathcal{D}_s, \; s = 1,\dots,N \right)
\end{equation}
\begin{equation}
\begin{aligned}
\frac{\overline{\phi}_N^{(s)}(t,Z_s^*)}{N-s+1} & + 
\underline{\phi}_N^{(s-1)}(t,(Z_s^*)^{(i)})
+ \underline{\phi}_N^{(s-1)}(t,(Z_s^*)^{(j)}) = \\
& =  \frac{\overline{\phi}_N^{(s)}(t,Z_s)}{N-s+1}
 + \overline{\phi}_N^{(s-1)}(t,Z_s^{(i)}) + 
\overline{\phi}_N^{(s-1)}(t,Z_s^{(j)})
\end{aligned}
\end{equation}
\begin{equation*}
\quad \quad \quad \quad \quad \quad \quad \quad \quad 
 \left(Z_s\in\left(\Sigma_s (i,j)\times \mathbb{R}^{d s}\right)
\cap \partial \mathcal{D}_s^{\textnormal{post}},
\; s = 2,\dots,N\right)
\end{equation*}
The lower envelope, in turn, solves the following equation:
\begin{equation}
\left( \partial_t - V_s \cdot \nabla_{X_s} \right) 
\underline{\phi}_N^{(s)} (t,Z_s) = 0
\qquad  \left( Z_s \in \mathcal{D}_s, \; s = 1,\dots,N \right)
\end{equation}
\begin{equation}
\begin{aligned}
\frac{\underline{\phi}_N^{(s)}(t,Z_s^*)}{N-s+1} & + 
\overline{\phi}_N^{(s-1)}(t,(Z_s^*)^{(i)})
+ \overline{\phi}_N^{(s-1)}(t,(Z_s^*)^{(j)}) = \\
& =  \frac{\underline{\phi}_N^{(s)}(t,Z_s)}{N-s+1}
 + \underline{\phi}_N^{(s-1)}(t,Z_s^{(i)}) + 
\underline{\phi}_N^{(s-1)}(t,Z_s^{(j)})
\end{aligned}
\end{equation}
\begin{equation*}
\quad \quad \quad \quad \quad \quad \quad \quad \quad 
 \left(Z_s\in\left(\Sigma_s (i,j)\times \mathbb{R}^{d s}\right)
\cap \partial \mathcal{D}_s^{\textnormal{post}},
\; s = 2,\dots,N\right)
\end{equation*}

We are able to prove the following result by elementary
manipulations:
\begin{lemma}
\label{lemma:Comp-1}
Let $1 \leq S \leq N$ and assume that, for
$1\leq s \leq S$,
\begin{equation}
\underline{\phi}_N^{(s)} (0) \leq
\phi_N^{(s)} (0) \leq \overline{\phi}_N^{(s)} (0)
\end{equation}
Then for all $t \geq 0$ and all $1\leq s \leq S$,
\begin{equation}
\underline{\phi}_N^{(s)} (t) \leq
\phi_N^{(s)} (t) \leq \overline{\phi}_N^{(s)} (t)
\end{equation}
\end{lemma}
\begin{proof}
Apply induction on $S$.
\end{proof}
\begin{remark}
Note that Lemma \ref{lemma:Comp-1} is a very general result because
the lower and upper envelopes still see cancellations and may not
be non-negative.
\end{remark}

Let us now define a new hierarchy which solves the following equation:
\begin{equation}
\label{eq:Comp-phihat-1}
\left( \partial_t - V_s \cdot \nabla_{X_s} \right) 
\hat{\phi}_{N}^{(s)} (t,Z_s) = 0
\qquad  \left( Z_s \in \mathcal{D}_s, \; s = 1,\dots,N \right)
\end{equation}
\begin{equation}
\label{eq:Comp-phihat-2}
\begin{aligned}
\frac{\hat{\phi}_N^{(s)}(t,Z_s^*)}{N-s+1} & - 
\hat{\phi}_N^{(s-1)}(t,(Z_s^*)^{(i)})
- \hat{\phi}_N^{(s-1)}(t,(Z_s^*)^{(j)}) = \\
& =  \frac{\hat{\phi}_N^{(s)}(t,Z_s)}{N-s+1}
 + \hat{\phi}_N^{(s-1)}(t,Z_s^{(i)}) + 
\hat{\phi}_N^{(s-1)}(t,Z_s^{(j)})
\end{aligned}
\end{equation}
\begin{equation*}
\quad \quad \quad \quad \quad \quad \quad \quad \quad 
 \left(Z_s\in\left(\Sigma_s (i,j)\times \mathbb{R}^{d s}\right)
\cap \partial \mathcal{D}_s^{\textnormal{post}},
\; s = 2,\dots,N\right)
\end{equation*}
We assume $\hat{\phi}_N^{(s)} (0) \geq 0$ for all $s$; then, the
same holds for all $t>0$.
Then if $\overline{\phi}_N^{(s)} (t) = \hat{\phi}_N^{(s)}(t)$ and
$\underline{\phi}_N^{(s)} (t) = -\hat{\phi}_N^{(s)}(t)$ for all $s$
then the functions $\overline{\phi}_N^{(s)} (t)$ and
$\underline{\phi}_N^{(s)} (t)$ solve the equations for upper and
lower envelopes, respectively. In particular we deduce the
following result:
\begin{lemma}
\label{lemma:Comp-2}
Let $1\leq S \leq N$ and assume that, for $1\leq s \leq S$,
\begin{equation}
\left| \phi_N^{(s)} (0) \right| \leq
\hat{\phi}_N^{(s)} (0)
\end{equation}
Then for all $t\geq 0$ and all $1\leq s \leq S$,
\begin{equation}
\left| \phi_N^{(s)} (t) \right| \leq
\hat{\phi}_N^{(s)} (t)
\end{equation}
\end{lemma}

\section{A Hierarchy of Singular Sets}
\label{sec:Sing}

The dynamics of the BBGKY hierarchy is determined by the initial
data $f_N^{(s)} (0)$ evaluated
along very singular subsets of the reduced phase space
$\mathcal{D}_s$. Our goal is to define these singular sets in
an \emph{abstract} fashion by relying on duality and
Lemma \ref{lemma:Comp-2}. We will see that by choosing initial
data $\hat{\phi}_N^{(s)} (0)$ carefully, we can force the dual
BBGKY hierarchy to identify the singular sets for us.

Let us define
\begin{equation}
\hat{\phi}_{N,j}^{(s)} (0) = 
\left\{
\begin{aligned}
& \mathbf{1}_{\mathcal{D}_s} \textnormal{ if } s=j \\
& 0 \textnormal{ otherwise}
\end{aligned}
\right.
\end{equation}
We let $\hat{\phi}_{N,j}^{(s)} (t)$ solve
(\ref{eq:Comp-phihat-1}-\ref{eq:Comp-phihat-2}). Note carefully
that
$\hat{\phi}_{N,j}^{(j+1)} (t)$ is, at each point, an integer 
multiple of $N-j$. Furthermore by \cite{BFK1998}, at any two points
where $\hat{\phi}_{N,j}^{(j+1)} (t)$ is non-zero, it is separated
by (at most) a \emph{constant multiple} depending on $j$. To a good
approximation, $\hat{\phi}_{N,j}^{(j+1)} (t)$ is a 
\emph{delta function} concentrated on a subset (in fact a
submanifold) of
$\mathcal{D}_{j+1}$. The easiest way to see this is that
the dual BBGKY hierarchy (as well as the hierarchy satisfied by
$\hat{\phi}_N^{(s)}$) is well-posed in some weighted
$\mathcal{L}^1$ space. \cite{De2017} The Lebesgue measure of the
support of $\hat{\phi}_N^{(s)} (t)$ is locally of order
$\mathcal{O} (N^{-1})$, which is 
$\mathcal{O}(\varepsilon^{d-1})$ due to the Boltzmann-Grad scaling.
This is consistent with the estimates provided in
\cite{BGSRS2016}.

All the above considerations remain valid when applied to
$\hat{\phi}_{N,j}^{(s)} (t)$ for $s > j+1$, except that we gain
a power of $N$ pointwise each time we increment $s$ by one. The
correct interpretation of this phenomenon is that we are
isolating submanifolds of increasing codimension, due to an
increasing number of collision constraints. In particular,
the measure of the support of
$\hat{\phi}_{N,j}^{(s)} (t)$, for $s>j$, is locally of
order $\mathcal{O}(\varepsilon^{(s-j)(d-1)})$, which is again
consistent with \cite{BGSRS2016}. To summarize, for each
$s\in\mathbb{N}$ and $N \gg s$, we have a hierarchy of singular
sets given by the support of $\hat{\phi}_{N,j}^{(s)} (t)$ for
$t$ large and $1 \leq j < s$. \emph{These are exactly the sets
which are relevant for the dynamics of the BBGKY hierarchy.}

In the remainder of this section we will quantify the above
considerations in a precise way. First we observe that the
functions $\hat{\phi}_{N,j}^{(s)} (t)$ are increasing in time.
\begin{lemma}
For any $1 \leq j < s$, any $0 < t < t^\prime$, and
almost every $Z_s \in \mathcal{D}_s$,
\begin{equation}
\hat{\phi}_{N,j}^{(s)} (t,Z_s) \leq
\hat{\phi}_{N,j}^{(s)} (t^\prime, Z_s)
\end{equation}
\end{lemma}
\begin{proof}
We write the Duhamel formula to express
$\hat{\phi}_{N,j}^{(s)} (t,Z_s)$ in terms of
$\hat{\phi}_{N,j}^{(s-1)}$; then apply induction on $s$.
The number of jumps in
the Duhamel formula is non-decreasing in time, and the size of each jump
is non-decreasing in time by the inductive hypothesis, hence the conclusion.
\end{proof}

The functions $\hat{\phi}_{N,j}^{(s)} (t)$ range over a discrete set.
\begin{lemma}
\label{lemma:Sing-DISC}
For any $1 \leq s < j$ and any $t>0$,
$\hat{\phi}_{N,j}^{(s)} (t) \equiv 0$; also, for any $t>0$,
$\hat{\phi}_{N,j}^{(j)} (t) \equiv \mathbf{1}_{\mathcal{D}_j}$.
For any $1 \leq j < s$, any $t>0$, and almost every
 $Z_s \in \mathcal{D}_s$,
\begin{equation}
0 \leq \hat{\phi}_{N,j}^{(s)} (t,Z_s)
\in (N-j)(N-j-1)\dots (N-s+1) \mathbb{Z}
\end{equation}
\end{lemma}
\begin{proof}
Apply induction on $s$.
\end{proof}

\begin{proposition}
\label{prop:Sing-UB}
For any $1 \leq j < s$, any $t>0$, and almost every
$Z_s \in \mathcal{D}_s$,
\begin{equation}
0 \leq \hat{\phi}_{N,j}^{(s)} (t,Z_s) \leq
\prod_{j < k \leq s} 
\left(
4 (N-k+1) \left( 32 k^{\frac{3}{2}} \right)^{k^2}
\right)
\end{equation}
\end{proposition}
\begin{proof}
Use induction on $s$ and the collision bound from \cite{BFK1998}.
Note carefully that the spatial domain is $\mathbb{R}^d$.
\end{proof}

\begin{corollary}
\label{cor:Sing-COMP}
For any $1 \leq j < s$, any $t>0$, almost every point
$\tilde{Z}_s \in \mathcal{D}_s$ such that
$\hat{\phi}_{N,j}^{(s)} (t,\tilde{Z}_s) \neq 0$, and almost every
$Z_s \in \mathcal{D}_s$,
\begin{equation}
0 \leq \hat{\phi}_{N,j}^{(s)} (t,Z_s) \leq
\hat{\phi}_{N,j}^{(s)} (t,\tilde{Z}_s)
\prod_{j < k \leq s}
 \left( 4 \left( 32 k^{\frac{3}{2}}\right)^{k^2} \right)
\end{equation}
\end{corollary}
\begin{proof}
Apply Lemma \ref{lemma:Sing-DISC} and
Proposition \ref{prop:Sing-UB}.
\end{proof}

Motivated by Corollary \ref{cor:Sing-COMP}, 
for $0 \leq k < s$ and $T>0$, let us define 
\begin{equation}
\mathcal{W}_s^k (T) =
\left\{ Z_s \in \mathcal{D}_s \left|
\hat{\phi}_{N,s-k}^{(s)} (T,Z_s) \neq 0 \right. \right\}
\end{equation}
We have that
$\mathcal{W}_s^k (T_1) \subset \mathcal{W}_s^k (T_2)$ whenever
$0 < T_1 < T_2$.

\begin{proposition}
\label{prop:Sing-SIZE}
There exists a constant $C_d > 0$ such that the following is true:
For any $0 < k < s$, any $\beta > 0$, and any $T>0$, in the
Boltzmann-Grad scaling $N\varepsilon^{d-1} = \ell^{-1}$, there holds
\begin{equation}
\begin{aligned}
& \int_{\mathbb{R}^{2ds}} \mathbf{1}_{\mathcal{W}_s^k (T)}
e^{-\beta \left( E_s (Z_s) + I_s (Z_s) \right)} dZ_s 
\leq
C(d,\beta)^{s+k} \ell^{-k} 
\frac{\prod_{q=0}^{k-1} (s-q)}
{\prod_{q=1}^k (N-s+q)}
\end{aligned}
\end{equation}
\end{proposition}
\begin{proof}
The proof of Theorem
 \ref{thm:dual-LWP-2} is easily adapted to apply to the
functions $\hat{\phi}_{N,j}^{(s)}$; then we use Lemma \ref{lemma:Sing-DISC}.
\end{proof}

\begin{remark}
Proposition \ref{prop:Sing-SIZE} implies that, for any bounded
set $B \subset \mathbb{R}^{2ds}$, the set
$B\cap \mathcal{W}_s^k (T)$  has measure of order
$\mathcal{O}(\varepsilon^{k(d-1)})$; a similar result was found
 in \cite{BGSRS2016}.
\end{remark}

\begin{proposition}
\label{prop:Sing-UB-2}
Consider the solution to the equations
(\ref{eq:Comp-phihat-1}-\ref{eq:Comp-phihat-2}) with initial
data
\begin{equation}
\hat{\phi}_N^{(s)} (0,Z_s) =
\left\{
\begin{aligned}
& N^k
\mathbf{1}_{\mathcal{W}_{s_1}^k (T)} \quad & \textnormal{ if }
s = s_1 \\
& 0  & \textnormal{ otherwise}
\end{aligned}
\right.
\end{equation}
Then for a.e. $t>0$, and a.e. $Z_s \in \mathcal{D}_s$, we have
\begin{equation}
0 \leq \hat{\phi}_N^{(s)} (t) \leq
\left\{
\begin{aligned}
& 0 & \textnormal{ if } s<s_1 \\
& C(K,s_1) N^{k+s-s_1} \mathbf{1}_{\mathcal{W}_{s}^{k+s-s_1} (T+t)} &
\textnormal{ if } s_1 \leq s \leq K \\
& \infty & \textnormal{ otherwise}
\end{aligned}
\right.
\end{equation}
\end{proposition}
\begin{proof}
It is enough to observe that $\hat{\phi}_N^{(s)} (0)$ is controlled
from above (up to a constant depending on $s_1$) by
$\hat{\phi}_{N,s_1-k}^{(s)} (T)$. The same control from above is
propagated to positive times, and then we can apply
Proposition \ref{prop:Sing-UB}.
\end{proof}

\begin{remark}
Proposition \ref{prop:Sing-UB-2} shows that the sets $\mathcal{W}_s^k (T)$
are not artifacts of our choice of constant functions for the initial
data. Indeed, even if we re-start the dual BBGKY flow with a singular
set as initial data, the further evolution is again concentrated on the
same family of singular sets.
\end{remark}

\section{Observables and Pseudo-Trajectories}
\label{sec:TR}

Pseudo-trajectories are fictitious trajectories which may be used
to express the solution of the BBGKY hierarchy in terms of the 
initial data; they are a standard tool in the analysis of BBGKY-type
hierarchies. \cite{L1975,GSRT2014} There is a close connection between
observables and pseudo-trajectories; for instance, we can see by the
fundamental duality relation (\ref{eq:duality-2}) that observables
allow us to express the solution of the BBGKY hierarchy in terms of
the initial data. In some sense, we may
 view observables as a functional representation of pseudo-trajectories.

Recall that $\psi_s^t : \mathcal{D}_s \rightarrow \mathcal{D}_s$ denotes
the flow of $s$ identical hard spheres of diameter $\varepsilon$.
We denote by $T_s (t)$ the operator which simply translates along
trajectories:
\begin{equation}
\left( T_s (t) f^{(s)} \right) (Z_s) =
f^{(s)} \left( \psi_s^{-t} Z_s \right)
\end{equation}
The BBGKY hierarchy may be written in integral form using Duhamel's
formula:
\begin{equation}
\label{eq:TR-duhamel-0}
f_N^{(s)} (t) = T_s (t) f_N^{(s)} (0) + (N-s) \varepsilon^{d-1} \int_0^t
T_s (t-t_1) C_{s+1} f_N^{(s+1)} (t_1) dt_1
\end{equation}
A similar expression can be obtained for $f_N^{(s+1)} (t_1)$ and plugged
into (\ref{eq:TR-duhamel-0}); 
this process can be repeated a finite number of times.
Ultimately one is able to express any marginal
$f_N^{(s)} (t)$ in terms of the initial data; that is,
\begin{equation}
\label{eq:TR-duhamel-1}
\begin{aligned}
& f_N^{(s)} (t) = \sum_{k=0}^{N-s} a_{N,k,s}
 \int_0^t \int_0^{t_1} \dots \int_0^{t_{k-1}} dt_k \dots dt_1 \times \\
& \qquad \qquad \qquad 
\times T_s (t-t_1) C_{s+1} T_{s+1} (t_1-t_2) \dots
C_{s+k} T_{s+k} (t_k) f_N^{(s+k)} (t_k) 
\end{aligned}
\end{equation}
where
\begin{equation}
\label{eq:TR-anks}
a_{N,k,s} = 
\frac{(N-s)!}{(N-s-k)!} \varepsilon^{k(d-1)}
\end{equation}
We will reformulate (\ref{eq:TR-duhamel-1}) using pseudo-trajectories.

Pseudo-trajectories define a fictitious dynamics involving a variable
number of particles. Let us be given $Z_s \in \mathcal{D}_s$,
a time $t>0$, creation times $0 < t_k \leq \dots \leq t_1 \leq t$, creation velocities
$v_{s+j} \in \mathbb{R}^d$, impact parameters $\omega_j \in \mathbb{S}^{d-1}$, and
indices $i_j \in \left\{ 1,2,\dots,s+j-1\right\}$. The symbol
\begin{equation}
Z_{s,s+k} \left[ Z_s,t; \left\{ t_j,v_{s+j},\omega_j,i_j \right\}_{j=1}^k \right]
\end{equation}
means that we start with the given particles $Z_s$,
and evolve \emph{backwards} under the $s$ particle flow $\psi_s^t$ for a time
$t-t_1$. Then we create a particle adjacent to the $i_1$ particle with velocity $v_1$ and
impact parameter $\omega_1$,
perform a collisional change of variables if needed, then continue the particle flow
backwards for a time $t_1 - t_2$, and so on. The end state, denoted
by $Z_{s,s+k} [\dots]$, is a configuration in
$\mathbb{R}^{d(s+k)} \times \mathbb{R}^{d(s+k)}$. Associated with
each pseudo-trajectory is an iterated collision kernel
\begin{equation}
b_{s,s+k} \left[ Z_s,t; \left\{ t_j,v_{s+j},\omega_j,i_j \right\}_{j=1}^k \right]
\end{equation}
 Particle creations are usually
referred to as \emph{collisions}, whereas collisions induced by the dynamics (not
creations) are usually called \emph{recollisions}. We refer to \cite{PSS2014,GSRT2014,De2017} for
precise definitions of psuedo-trajectories.

\begin{remark}
The pseudo-trajectory is not well-defined for impact paramters which would force
particles to overlap at the time of particle creation.
\end{remark}

We may now re-write (\ref{eq:TR-duhamel-1}) in the following form:
\begin{equation}
\begin{aligned}
& f_N^{(s)} (t,Z_s) = \sum_{k=0}^{N-s} a_{N,k,s} \times \\
& \times \sum_{i_1 = 1}^s \dots \sum_{i_k = 1}^{s+k-1} \int_0^t \dots
\int_0^{t_{k-1}} \int_{\mathbb{R}^{dk}} 
\int_{(\mathbb{S}^{d-1})^k} \left( \prod_{m=1}^k d\omega_m dv_{s+m} dt_m
\right) \times \\
& \times \left( b_{s,s+k} [\cdot] f_N^{(s+k)} (0,Z_{s,s+k} [\cdot]) \right)
\left[ Z_s,t; \left\{ t_j,v_{s+j},\omega_j,i_j \right\}_{j=1}^k \right]
\end{aligned}
\end{equation}
Let $E \subset \mathcal{D}_s$ be a bounded measurable set which is symmetric
under interchange of particle indices. Then we may write
\begin{equation}
\label{eq:TR-duhamel-E}
\begin{aligned}
& \int_{\mathcal{D}_s} \mathbf{1}_{Z_s \in E} f_N^{(s)} (t,Z_s) dZ_s =
\sum_{k=0}^{N-s} a_{N,k,s} \int_{\mathcal{D}_s} dZ_s \mathbf{1}_{Z_s \in E}
 \times \\
& \times \sum_{i_1 = 1}^s \dots \sum_{i_k = 1}^{s+k-1} \int_0^t \dots
\int_0^{t_{k-1}} \int_{\mathbb{R}^{dk}} 
\int_{(\mathbb{S}^{d-1})^k} \left( \prod_{m=1}^k d\omega_m dv_{s+m} dt_m
\right) \times \\
& \times \left( b_{s,s+k} [\cdot] f_N^{(s+k)} (0,Z_{s,s+k} [\cdot]) \right)
\left[ Z_s,t; \left\{ t_j,v_{s+j},\omega_j,i_j \right\}_{j=1}^k \right]
\end{aligned}
\end{equation}
Let $\Phi_{N,E}$ be the solution of the dual BBGKY hierarchy with 
initial data
\begin{equation}
\phi_{N,E}^{(s^\prime)} (0,Z_{s^\prime}) = 
\left\{
\begin{aligned}
& \mathbf{1}_{Z_{s^\prime} \in E} \qquad  & \textnormal{ if } s^\prime = s \\
& 0 & \textnormal{ otherwise}
\end{aligned} 
\right.
\end{equation}
Using (\ref{eq:duality-1}-\ref{eq:duality-2}) we may now write
\begin{equation}
\label{eq:TR-duality-E}
\begin{aligned}
& \frac{1}{s!} \int_{\mathcal{D}_s}
\mathbf{1}_{Z_s \in E} f_N^{(s)} (t,Z_s) dZ_s  \\
& \qquad \qquad \qquad  =
\sum_{k=0}^{N-s} \frac{1}{(s+k)!}
\int_{\mathcal{D}_{s+k}} \phi_{N,E}^{(s+k)} (t,Z_{s+k}) 
f_N^{(s+k)} (0,Z_{s+k}) dZ_{s+k}
\end{aligned}
\end{equation}
Comparing (\ref{eq:TR-duhamel-E}) and (\ref{eq:TR-duality-E}), and recalling
that $F_N (0)$ is an arbitrary sequence of symmetric functions, we conclude
the following identity:
\begin{equation}
\label{eq:TR-duality-E-2}
\begin{aligned}
&   \phi_{N,E}^{(s+k)} (t,Z_{s+k})
=\sum_\alpha 
 a_{N,k,s}  b_{s,s+k} \left[ Z_s^{\alpha},t, 
\left\{ t_j^\alpha,v_{s+j}^\alpha,\omega_j^\alpha,i_j^\alpha
\right\}_{j=1}^k \right] 
\times \\
& \qquad \qquad \qquad \qquad 
\times \left| \det \frac{\partial Z_{s,s+k} \left[ Z_s^\alpha,t;
\left\{ t_j^\alpha,v_{s+j}^\alpha,\omega_j^\alpha,i_j^\alpha
 \right\}_{j=1}^k\right]}
{\partial Z_s \partial t_1\dots \partial t_k
\partial v_{s+1} \dots \partial v_{s+k} \partial \omega_1 \dots \partial
\omega_k } \right|^{-1}
\end{aligned}
\end{equation}
The (finite) sum $\sum_\alpha$ runs over all pseudo-trajectories that initially
contain $s$ particles with $Z_s^\alpha \in E$,
and end at a point (with $s+k$ particles) of the form $\sigma Z_{s+k}$ for some
$\sigma \in \mathcal{S}_{s+k}$.\footnote{Note that two configurations
$\sigma_1 Z_{s+k}$ and $\sigma_2 Z_{s+k}$ such that
 $\sigma_1^{-1} \sigma_2$ leaves fixed
the last $k$ particle indices will correspond to physically indistinguishable
 collections of 
pseudo-trajectories. We do not double-count in this case.}
Here $\mathcal{S}_s$ is the symmetric group
on $s$ letters which acts by interchange of particle indices; namely,
$(\sigma {Z_s})_j = z_{\sigma (j)}$.

Comparing (\ref{eq:TR-duality-E-2}) with 
(\ref{eq:BBGKY-dual-1}-\ref{eq:BBGKY-dual-2}), we realize that each
term appearing in the sum $\sum_\alpha$ in
(\ref{eq:TR-duality-E-2}) is associated with a contribution to
$\phi_{N,E}^{(s+k)}$ equal to
\begin{equation}
\label{eq:TR-duality-3}
\pm (N-s-k+1)(N-s-k+2)\dots (N-s)
\end{equation}
Hence, in view of (\ref{eq:TR-anks}), we must have
\begin{equation}
\label{eq:TR-PTR-1}
\begin{aligned}
& \left| \det \frac{\partial Z_{s,s+k} \left[ Z_s^\alpha,t;
\left\{ t_j^\alpha,v_{s+j}^\alpha,\omega_j^\alpha,i_j^\alpha
 \right\}_{j=1}^k\right]}
{\partial Z_s \partial t_1\dots \partial t_k
\partial v_{s+1} \dots \partial v_{s+k} \partial \omega_1 \dots \partial
\omega_k } \right| \\
& \qquad \qquad \qquad \qquad = \varepsilon^{k(d-1)}
\left| b_{s,s+k} \left[ Z_s^{\alpha},t, 
\left\{ t_j^\alpha,v_{s+j}^\alpha,\omega_j^\alpha,i_j^\alpha
\right\}_{j=1}^k \right] \right| 
\end{aligned}
\end{equation}
We remark that (\ref{eq:TR-PTR-1}) appears in \cite{PS2015II} as the
result of a direct computation. The above
argument supplies an alternative approach to deriving (\ref{eq:TR-PTR-1}).

Finally we note that a formula similar to (\ref{eq:TR-duality-E-2}) is
also available for the functions $\hat{\phi}_{N,j}^{(s)}$ defined in
Section (\ref{sec:Sing}). In particular, we easily deduce an
alternative definition for the sets $\mathcal{W}_s^k (T)$. Let us
define
\begin{equation}
\mathcal{V}_s^k (T) = \left\{ Z_s \in \mathcal{D}_s \left|
\begin{aligned}
& \exists Z_{s-k} \in \mathcal{D}_{s-k}, 0< t_k < \dots < t_1 < t < T, \\
& v_{s+1},\dots,v_{s+k},\omega_1,\dots,\omega_k,
i_1,\dots,i_k, \\
& \textnormal{such that} \\
& \qquad Z_s = Z_{s-k,s} \left[ Z_{s-k},t; \left\{ t_j,v_{s+j},\omega_j,i_j
\right\}_{j=1}^k \right]
\end{aligned}
\right.
\right\}
\end{equation}
Then we have
\begin{equation}
\mathcal{W}_s^k (T) = \bigcup_{\sigma \in \mathcal{S}_{s}}
\sigma \mathcal{V}_s^k (T)
\end{equation}
Notice that for any bounded set $B \subset \mathbb{R}^{2ds}$ which is symmetric
under particle interchange, we have
\begin{equation}
\left| B \cap \mathcal{V}_s^k (T) \right| \leq
\left| B \cap \mathcal{W}_s^k (T) \right| \leq
(s+k)! \left| B \cap \mathcal{V}_s^k (T) \right|
\end{equation}
Hence the estimate from Proposition \ref{prop:Sing-SIZE} (in particular the
remark immediately following the proposition) applies equally well to
either $\mathcal{V}_s^k (T)$ or $\mathcal{W}_s^k (T)$.

\section{Propagation of Chaos}
\label{sec:PR}

We will require two auxiliary sets to state our main result.
These sets appear in our proof for technical reasons and could
actually be removed from the main theorem.
Note that $\eta > 0$ is a small parameter which can depend
on $\varepsilon$.
\begin{equation}
\mathcal{K}_s = \left\{ Z_s = (X_s,V_s) \in \mathcal{D}_s \left|
\forall \tau > 0,\;
\psi_s^{-\tau} Z_s = (X_s -V_s \tau, V_s)
\right. \right\}
\end{equation}
\begin{equation}
\mathcal{U}_s^\eta = \left\{ Z_s \in \mathcal{D}_s \left|
\inf_{1\leq i < j \leq s} |v_i - v_j|>\eta \right. \right\}
\end{equation}
Using the sets defined above, we introduce an $L^1$ seminorm defined as follows:
\begin{equation}
\label{eq:seminorm-1}
\left\Vert f^{(s)} \right\Vert_{\varepsilon,s,k,\eta,T^\prime,R}
= \varepsilon^{-k(d-1)}
\left\Vert f^{(s)} \mathbf{1}_{\mathcal{K}_s  \cap \mathcal{U}_s^\eta
\cap \mathcal{V}_s^k (T^\prime)} \mathbf{1}_{(E_s +I_s)(Z_s) \leq R^2} \right\Vert_{L^1_{Z_s}}
\end{equation}
Note carefully $\left\Vert f^{(s)} \right\Vert_{\varepsilon,s,k,\eta,T^\prime,R}
\leq C(s,k,T^\prime,R) \left\Vert f^{(s)} \right\Vert_\infty$. In particular, since the
constant does not depend on $\varepsilon$,
\emph{any estimate which can be carried out in $L^\infty$ requires no further comment
in our new topology}.
 For this reason, most of the developments of
\cite{De2017} can be carried over trivially. We will discuss only the differences which arise
in our proof (which actually amounts to just one error term).

We are now ready to state our main result.

\begin{theorem}
\label{thm:chaos-1}
Suppose that the Boltzmann equation (\ref{eq:IN-boltz}) has a non-negative solution
$f(t)$ for $t\in [0,T]$, with $\int f(t) dx dv = 1$. Also, suppose there exists
$\beta_T > 0$ such that
\begin{equation}
\sup_{0\leq t \leq T} \sup_{x\in \mathbb{R}^d} \sup_{v\in \mathbb{R}^d}
e^{\frac{1}{2} \beta_T |v|^2} f(t,x,v) < \infty
\end{equation}
and $f(t) \in W^{1,\infty} (\mathbb{R}^d \times \mathbb{R}^d)$ for $t\in [0,T]$.
Let $f_N (0)$ be a symmetric probability density on $\mathcal{D}_N$ for each
$N\in \mathbb{N}$, and enforce the Boltzmann-Grad scaling $N\varepsilon^{d-1}=\ell^{-1}$.
Suppose that there exists a $\tilde{\beta}_T > 0$,
$\tilde{\mu}_T \in \mathbb{R}$ such that the marginals $f_N^{(s)}$ of $f_N$
satisfy
\begin{equation}
\sup_{N\in \mathbb{N}} \sup_{0\leq t \leq T}
\sup_{1\leq s \leq N} \sup_{Z_s \in \mathcal{D}_s}
e^{\tilde{\beta}_T E_s (Z_s)} e^{\tilde{\mu}_T s}
\left| f_N^{(s)} (t,Z_s) \right| < \infty
\end{equation}
Let $\eta (\varepsilon) = \sqrt{\varepsilon}$.
Assume that
for all $s\in \mathbb{N}$, $0\leq k < s$, $T^\prime > 0$, and $R>0$, we have
\begin{equation}
\limsup_{N\rightarrow \infty} \left\Vert f_N^{(s)} (0) - f^{\otimes s} (0) 
\right\Vert_{\varepsilon,s,k,\eta(\varepsilon),T^\prime,R} = 0
\end{equation}
Then for all $s\in \mathbb{N}$, $0\leq k < s$, $T^\prime > 0$, $R>0$, and
$t\in [0,T]$, we have
\begin{equation}
\limsup_{N\rightarrow \infty} \left\Vert f_N^{(s)} (t) - f^{\otimes s} (t) 
\right\Vert_{\varepsilon,s,k,\eta(\varepsilon),T^\prime,R} = 0
\end{equation}
\end{theorem}

\subsection{An Example}

We are going to construct sequences of initial data which satisfy the
conditions of Theorem \ref{thm:chaos-1}, but which do not converge uniformly
as required by Lanford's original proof. \cite{L1975} 
We begin with a sequence of functions
$0 \leq f_{0,N} : \mathbb{R}^d \times\mathbb{R}^d
 \rightarrow \mathbb{R}$ such that, for some $\beta_0 > 0$,
\begin{equation}
\sup_{N \in \mathbb{N}} \sup_{x \in \mathbb{R}^d}
\sup_{v \in \mathbb{R}^d}
e^{\frac{1}{2} \beta_0 (|x|^2+|v|^2)} f_{0,N} (x,v) < \infty
\end{equation}
We also assume
\begin{equation}
\int_{\mathbb{R}^d \times \mathbb{R}^d} f_{0,N} (x,v) dx dv = 1
\end{equation}
We furthermore assume that there exists some
$f_0$ in the Schwartz class such that
\begin{equation}
\lim_{N \rightarrow \infty}
\left\Vert f_{0,N} - f_0 \right\Vert_{L^1_{x,v}} = 0
\end{equation}
Finally, we assume there exists a decreasing sequence of open
balls $B_N \subset \mathbb{R}^d \times \mathbb{R}^d$ such
that $\cap_N B_N = \varnothing$ and
\begin{equation}
\liminf_{N\rightarrow \infty}
\inf_{(x,v) \in B_N}
\left( f_{0,N} (x,v) - f_0 (x,v) \right) > 0
\end{equation}
We also assume that the Lebesgue measure of $B_N$ is
at least $C(\log N)^{-1}$ for some constant $C$.

The $N$-particle data is then defined as
\begin{equation}
f_N (Z_N) = \mathcal{Z}_N^{-1} f_{0,N}^{\otimes N} (Z_N)
\mathbf{1}_{Z_N \in \mathcal{D}_N}
\end{equation}
where
\begin{equation}
\mathcal{Z}_N = \int f_{0,N}^{\otimes N} (Z_N) 
\mathbf{1}_{Z_N \in \mathcal{D}_N} dZ_N
\end{equation}
is a normalization factor. In the Boltzmann-Grad scaling
$N\varepsilon^{d-1}=\ell^{-1}$, 
it is possible to show that for all $s\in \mathbb{N}$ there
holds:
\begin{equation}
\label{eq:EX-1}
\limsup_{N\rightarrow \infty} \left\Vert
\mathbf{1}_{Z_s \in \mathcal{D}_s}
\left( f_N^{(s)} (0,Z_s) - f_{0,N}^{\otimes s} (Z_s)
\right) \right\Vert_{L^\infty_{Z_s}} = 0
\end{equation}
(See \cite{GSRT2014} or \cite{De2017} for a proof.)
The functions $f_{0,N}^{\otimes s}$ do not converge uniformly
on compact sets, so by (\ref{eq:EX-1}), the functions
$f_N^{(s)} (0)$ also do not converge uniformly on compact sets.
Therefore the hypotheses of Lanford's theorem are \emph{not}
fulfilled. 

We are going to show that the conditions of
Theorem \ref{thm:chaos-1} are satisfied, using the triangle
inequality. Indeed, for any $Z_s \in \mathcal{D}_s$, we have
\begin{equation}
\label{eq:EX-2}
\begin{aligned}
& \left| f_N^{(s)} (0,Z_s) - f_0^{\otimes s} (Z_s) \right| \leq
\left| f_N^{(s)} (0,Z_s) - f_{0,N}^{\otimes s} (Z_s)\right| + \\
&\qquad \qquad \qquad  + \sum_{i=1}^s
f_{0,N}^{\otimes (i-1)} (Z_{1:(i-1)})
\left| f_{0,N} (z_i) - f_0 (z_i) \right|
f_0^{\otimes (s-i)} (Z_{(i+1):s})
\end{aligned}
\end{equation}
The first term on the right hand side of (\ref{eq:EX-2}) is estimated
using (\ref{eq:EX-1}) and the bound
\begin{equation}
\left\Vert f^{(s)} \right\Vert_{\varepsilon,s,k,\eta,T^\prime,R}
\leq C(s,k,T^\prime,R) \left\Vert f^{(s)}
\right\Vert_{\infty}
\end{equation}
The remaining terms are all estimated in the same way so we
only consider the case $i=1$. We are trying to bound
\begin{equation}
\left\Vert \left|
f_{0,N}(z_1) - f_0 (z_1) \right|
f_0^{\otimes (s-1)} (Z_{2:s}) \right\Vert_{\varepsilon,s,k,
\eta(\varepsilon),T^\prime,R}
\end{equation}
Since $f_0 \leq C e^{-\frac{1}{2} \beta_0 (|x|^2+|v|^2)}$, we are
left with
\begin{equation}
C^s \left\Vert \left|
f_{0,N}(z_1) - f_0 (z_1) \right|
e^{-\beta_0 (E_{s-1}+I_{s-1}) (Z_{2:s})} \right\Vert_{\varepsilon,s,k,
\eta(\varepsilon),T^\prime,R}
\end{equation}
Now the point is that the condition $Z_s \in \mathcal{V}_s^k (T^\prime)$
can be interpreted as saying that $z_1$ is chosen
\emph{arbitrarily}, and then the $k$ collision constraints are simply
constraints imposed on the \emph{remaining}
particles. Since we have an $L^\infty$ bound in all but the
first coordinate, we gain a factor of $\varepsilon^{d-1}$ for each
of the $k$ collision constraints, so that
\begin{equation}
\begin{aligned}
& \left\Vert \left| f_{0,N} (z_1) - f_0 (z_1) \right|
e^{-\beta_0 (E_{s-1}+I_{s-1}) (Z_{2:s})} 
\right\Vert_{\varepsilon,s,k,\eta(\varepsilon),T^\prime,R}
\leq \\
& \qquad \qquad \qquad \qquad \qquad \qquad \qquad
\leq C(s,k,T^\prime,R) \left\Vert 
f_{0,N} - f_0 \right\Vert_{L^1_{x,v}}
\end{aligned}
\end{equation}
which yields the desired bound.

\begin{remark}
Note that the conclusion of Theorem \ref{thm:chaos-1} implies that
the first marginal $f_N^{(1)} (t)$ converges to $f (t)$ in the
norm topology of $L^1_{x,v}$. Similarly, as the example illustrates,
the hypotheses of the theorem are related to $L^1$ convergence of the
first marginal at $t=0$. Hence, Theorem \ref{thm:chaos-1} provides
a microscopic interpretation for $L^1$ convergence in the
Boltzmann-Grad limit.
\end{remark}

\subsection{Proof of Theorem \ref{thm:chaos-1}}

Following \cite{De2017}, we may assume that, for some $\beta_0 > 0$,
$\mu_0 \in \mathbb{R}$, there holds
\begin{equation}
\sup_{1\leq s \leq N} \sup_{Z_s \in \mathcal{D}_s} 
\left| f_N^{(s)} (0,Z_s) \right| e^{\beta_0 (E_s+I_s) (Z_s)}
e^{\mu_0 s} \leq 1
\end{equation}
\begin{equation}
\sup_{1\leq s \leq N} \sup_{Z_s \in \mathcal{D}_s} 
\left| f^{\otimes s} (0,Z_s) \right| e^{\beta_0 ( E_s+I_s) (Z_s)}
e^{\mu_0 s} \leq 1
\end{equation}
The following error estimate was proven on a small time interval $[0,T_L]$,
using suitable truncations in $L^\infty$: (see \cite{De2017} for details)
\begin{equation}
\label{eq:CONV-est}
\begin{aligned}
& \left| \left( f_N^{(s)} - f^{\otimes s}\right) (t,Z_s) \right|
\mathbf{1}_{Z_s \in \mathcal{K}_s  \cap \mathcal{U}_s^{\eta}}
\mathbf{1}_{(E_s+I_s) (Z_s) \leq 2 R^2} \\
& \leq 3 e^{-(\mu_0 - 2)s} \left( e^{-\frac{1}{2} \beta_0 R^2} + e^{-n}\right) +\\
& + \left[ 1-\left(1-\frac{n}{N}\right)^n\right] e^{-\mu_0 s}
e^{C_d \ell^{-1} n R^{d+1} e^{-\mu_0} T_L} + \\
& + 2 e^{-\mu_0 s} n^2 \mathcal{A}_{n,R} \left[ \alpha + \frac{y}{\eta T_L}
+ C_{d,\alpha} \left( \frac{\eta}{R}\right)^{d-1} +
C_{d,\alpha} \theta^{(d-1)/2} \right] + \\
& + C_d n^{\frac{5}{2}} R^{-1} e^{|\mu_0| n} \varepsilon
e^{C_d \ell^{-1} n R^{d+1} e^{-\mu_0} T_L} + \\
& + C_d n^2 \varepsilon e^{C_d \ell^{-1} n R^{d+1} e^{-\mu_0} T_L}
\sup_{\substack{1\leq j \leq n \\ Z_j \in  \mathbb{R}^{2dj}}}
\left| \nabla_{Z_j} \left( f^{\otimes j}\right) (0,Z_j) \right|_2
\mathbf{1}_{(E_j+I_j) (Z_j) \leq 2 R^2} + \\
& + C_d e^{C_d \ell^{-1} n R^{d+1} e^{-\mu_0} T_L} \times \\
& \qquad \qquad \quad \times
\sup_{\substack{1\leq j \leq n \\ Z_j \in 
\mathbb{R}^{2dj}}}
\left| \left( f_N^{(j)} - f^{\otimes j} \right)(0,Z_j)\right|
\mathbf{1}_{Z_j \in \mathcal{K}_j \cap \mathcal{U}_j^\eta}
\mathbf{1}_{(E_j+I_j) (Z_j) \leq 2 R^2} \\
& = I + II + III + IV + V + VI
\end{aligned}
\end{equation}
We care only about the last term, $VI$, which is (ignoring the prefactor):
\begin{equation}
\sup_{\substack{1\leq j \leq n \\ Z_j \in \mathbb{R}^{dj}\times
\mathbb{R}^{dj}}}
\left| \left( f_N^{(j)} - f^{\otimes j} \right)(0,Z_j)\right|
\mathbf{1}_{Z_j \in \mathcal{K}_j  \cap \mathcal{U}_j^\eta}
\mathbf{1}_{(E_j+I_j) (Z_j) \leq 2 R^2}
\end{equation}
It is this last term which we desire to estimate in $L^1$ instead of
$L^\infty$. 
From the proof of \cite{De2017},
we are free to replace $VI$ in (\ref{eq:CONV-est}) by $VI^\prime$ where
\begin{equation}
\label{eq:CONV-VIprime}
\begin{aligned}
& VI^\prime = \\
& = \sum_{j=0}^{n-s} \sum_{i_1 = 1}^s \dots \sum_{i_j = 1}^{s+j-1} \ell^{-j}
\int_0^t \dots \int_0^{t_{j-1}} \int_{\left(B_{2R}^d\right)^j}
\int_{\left(\mathbb{S}^{d-1}\right)^j}
\prod_{m=1}^j d\omega_m dv_{s+m} dt_m \times \\
& \times \left( 1-\mathbf{1}_{\mathcal{B}_j}\right)  \times \\
& \times
\left( 
\begin{aligned}
& \left|b_{s,s+j} [\cdot] \right|
\mathbf{1}_{(E_{s+j}+I_{s+j}) (Z_{s,s+j} [\cdot]) \leq 2R^2} \times \\
& \qquad \qquad \qquad \qquad \quad \times
\left| f_{N}^{(s+j)} (0,Z_{s,s+j}[\cdot]) - f^{\otimes (s+j)}(0,Z_{s,s+j}[\cdot])
 \right| 
\end{aligned}
\right)\\
& \qquad \qquad \qquad \qquad \qquad \qquad\qquad \qquad \qquad
 \left[ Z_s, t; \left\{ t_r,v_{s+r},\omega_r,i_r\right\}_{r=1}^j \right]
\end{aligned}
\end{equation}
The set $\mathcal{B}_j$ specifies a ``bad set'' of particle creations which may
lead to recollisions. In particular, $Z_{s,s+j} \left[ \cdot \right]$ is in
$\mathcal{K}_{s+j}  \cap \mathcal{U}_{s+j}^\eta$.

Recall that
\begin{equation}
\left\Vert f^{(s)} \right\Vert_{\varepsilon,s,k,\eta,T^\prime,R} \leq
C(s,k,T^\prime,R)
\left\Vert f^{(s)} \right\Vert_\infty
\end{equation}
so the terms $I$ through $V$ in (\ref{eq:CONV-est}) are disposed with easily.
We only have to estimate 
\begin{equation}
\left\Vert VI^\prime \right\Vert_{\varepsilon,s,k,\eta,T^\prime,R}
\end{equation}
We can do that, in fact, because pseudo-trajectories obey a determinant
identity (see (\ref{eq:TR-PTR-1}) or \cite{PS2015II}):
\begin{equation}
\label{eq:CONV-det}
\begin{aligned}
& \left| \det
\frac{\partial Z_{s,s+j} \left[ Z_s,t;
\left\{ t_r,v_{s+r},\omega_r,i_r\right\}_{r=1}^j \right]}
{\partial Z_s \partial t_1 \dots \partial t_j \partial v_{s+1} \dots \partial v_{s+j}
\partial \omega_1 \dots \partial \omega_j}
\right| = \\
& \qquad \qquad \qquad \qquad \qquad
 = \varepsilon^{j(d-1)} \left| b_{s,s+j} \left[ Z_s,t;
\left\{ t_r,v_{s+r},\omega_r,i_r\right\}_{r=1}^j \right] \right|
\end{aligned}
\end{equation}
Each additional collision integral in (\ref{eq:CONV-VIprime})
 brings down an extra power of
$\varepsilon^{-(d-1)}$ according to (\ref{eq:CONV-det}), but this is exactly compensated
by an extra $\varepsilon^{d-1}$ coming from the definition of the norm
$\left\Vert f^{(s+j)} \right\Vert_{\varepsilon,s+j,k+j,\eta,T^\prime,R}$. This is because particle
creation maps $\mathcal{V}_{s+j}^{k+j} (T^\prime)$ to $\mathcal{V}_{s+j+1}^{k+j+1} (T^\prime)$.
In fact the only difficulty is that the map $Z_{s,s+k} [\dots]$ is not injective,
but the lack of injectivity can be quantified since there are \emph{no recollisions}
for any pseudo-trajectory appearing in our estimate (at least along the time interval
of interest).

Altogether we can write (up to combinatorial constants)
\begin{equation}
\begin{aligned}
\left\Vert VI^\prime \right\Vert_{\varepsilon,s,k,\eta,T^\prime-t,R}
\lesssim
\sum_{j=0}^{n-s}
\left\Vert 
f_N^{(s+j)} (0) - f^{\otimes (s+j)} (0)
\right\Vert_{\varepsilon,s+j,k+j,\eta,T^\prime,R}
\end{aligned}
\end{equation}
which tends to zero as $N\rightarrow \infty$, by assumption. Taking limits as
in \cite{De2017}, we are able to conclude.

\bibliography{duality}

\end{document}